\pdfoutput=1

\documentclass[12pt, a4paper]{article}
\usepackage{amsmath, amsthm, amssymb, enumitem, graphicx, color}
\usepackage[square,numbers]{natbib}
\allowdisplaybreaks

\newtheorem*{theorem}{Theorem 1}
\newtheorem*{corollary}{Corollary}
\newtheorem{defin}{Definition}
\newtheorem{lemma}{Lemma}

\newcommand{\bZ}{\mathbb{Z}}

\newcommand{\vol}{\operatorname{vol}}

\begin{document}
\author{I. Beach and R. Rotman}
\title{The Length of the Shortest Closed Geodesic on a Surface of Finite Area}
\maketitle

\begin{abstract} 
	In this paper we prove new upper bounds for the length of a shortest
	closed geodesic, denoted $l(M)$, on a complete, non-compact Riemannian surface 
	$M$ of finite area $A$. We will show that $l(M) \leq 4\sqrt{2A}$ on a manifold with one end,
	 thus improving the prior estimate of 
	C. B. Croke, who first established that $l(M) \leq 31 \sqrt{A}$. Additionally, for a surface with at least two ends we show that
	$l(M) \leq 2\sqrt{2A}$, improving the prior estimate of
	Croke that $l(M) \leq (12+3\sqrt{2})\sqrt{A}$.
\end{abstract}

\section{Introduction}
	The purpose of this paper is to establish two new upper bounds for the length 
	of a shortest periodic (i.e., closed) geodesic on a complete, non-compact Riemannian surface $M$ of a finite area. We denote this length by $l(M)$.
	\par
	While the existence of a periodic geodesic on a closed Riemannian manifold 
	is a consequence of the topology of the manifold (first shown by A. Fet and L. Lusternik in 1951 in the general case), simple examples show
	that even a non-trivial fundamental group is not enough to guarantee the existence of a periodic geodesic in the non-compact case. Thus, to prove the existence 
	of a periodic geodesic on a complete, non-compact Riemannian manifold, one needs
	some combination of additional topological and geometric constraints. For example, V. Bangert asked whether on a complete, non-compact Riemannian manifold of finite volume there always exists at least one periodic geodesic (cf. \cite{burns2013}). While 
	this question was fully answered in the positive for surfaces-- first 
	by G. Thorbergsson \cite{thorbergsson1978} for surfaces with at least three ends, and later by V. Bangert \cite{bangert1980} for surfaces with one or two ends-- only partial answers to this question  are known in higher dimensions (see \cite{assellemazzucchelli2016, bencigiannoni1991}). 
	\par
	Once the existence of a periodic geodesic on a Riemannian manifold is established, it is natural to ask how the length of a shortest such geodesic compares to other parameters of the manifold, such as its volume or, in the compact case, its diameter. For example, M. Gromov has asked if there exists a uniform constant $c(n)$ such that $l(M) \leq c(n) \vol(M)^\frac{1}{n}$ (see \cite{Gromov1983}).	Here $n$ is the dimension of the manifold and $\vol(M)$ is its volume. While this question remains open in dimensions greater than two, it has been resolved in the case of surfaces. For a survey of early results in dimension two, one should see the paper \cite{CrokeKatz2003} by C. B. Croke and M. Katz. 
	\par
	In this article we will improve the prior bound of Croke for the length
	of a shortest closed geodesic on a complete, non-compact surface of finite 
	area $A$ from $31 \sqrt{A}$ (see \cite{croke1988}) to $4\sqrt{2A}$. In particular, we improve the bound in the worst two cases. The most difficult case is that of a manifold that is topologically a once-punctured sphere. Secondly, we improve the bound on a complete, non-compact Riemannian manifold homeomorphic to a sphere with two punctures. Our improvements are due to the following modifications:
	an optimal use of the co-area formula as in \cite{rotman2006}; 
	utilizing min-max techniques on the space of $1$-cycles instead of on the space of closed, piecewise differentiable curves on $M$, as in \cite{calabi1992,nabutovsky2002,rotman2006,sabourau2004}; 
	and using a length shortening technique on geodesic nets (which we realize as pairs of loops) as in \cite{rotman2011,rotman2019}. 
	Note that the main difficulty of proving the existence of (short) periodic geodesics on complete, non-compact manifolds manifests itself in the curve ``running away to infinity'' under the length shortening flow. 
	\par
	Finally, we expect that the new bound we obtain is still not optimal. For the cases where the surface is a sphere with three or fewer punctures, we expect that the optimal bound will be the same as that for a Riemannian 2-sphere of area $A$, which is conjectured to be $(12)^\frac{1}{4} \sqrt{A}$ (cf. \cite{calabi1992,croke1988}). 
	
\section{Results}
	
	We will show that every complete, non-compact surface of finite area $A$ admits a closed geodesic whose length is bounded by a constant multiple of $\sqrt{A}$. Such surfaces can be partially classified by the number of ``ends'' they have, which we define as follows.
	
	\begin{defin}
		Let $M$ be a complete surface. If $M$ is homeomorphic to a compact surface with $n$ punctures, then we say that $M$ has $n$ ends.
	\end{defin}

	As above, suppose $M$ has area $A$ and $n$ ends, and let $l(M)$ be the length of a shortest  closed geodesic in $M$. It was shown by Hebda \cite{hebda1982} and Burago and Zalgaller \cite{burago1980} that $l(M)\leq\sqrt{2A}$ for every orientable, non-simply connected, compact surface. It was shown by Croke in \cite{croke1988} that $l(M)\leq31\sqrt{A}$ in the case that $M$ is a sphere or a once-punctured sphere. For the spherical case, this estimate was further developed in \cite{calabi1992}, and has since been refined to $l(M)\leq 4\sqrt{2A}$ \citep{nabutovsky2002, rotman2006, sabourau2004}. We prove that this sharper estimate also holds for the punctured sphere. For surfaces with two ends, we refine the current best known bound of $l(M)\leq(12+3\sqrt{2})\sqrt{A}$, also provided by Croke \cite{croke1988}, to $l(M)\leq 2\sqrt{2A}$. Our methods also provide an identical bound for surfaces with more than two ends, but in these cases this bound has already been proved by Croke \cite{croke1988} (see also \cite{bangert1980}).
	\par 
	Our main result is summarized as follows.
	
	\begin{theorem}
		Suppose $M$ is a complete, orientable surface with finite area $A$ and $n$ ends. Let $l(M)$ be the length of a shortest closed geodesic on $M$.
		\begin{enumerate}
			\item If $n\leq1$, then $l(M)\leq 4\sqrt{2A}$.
			\item If $n\geq2$, then $l(M)\leq 2\sqrt{2A}$.
		\end{enumerate}
	\end{theorem}

	The compact case ($n=0$) is dealt with in the literature as described above. Therefore we only consider surfaces with at least one end. As noted by Croke in \cite{croke1988}, each end of a surface with finite area contains a convex neighbourhood of infinity bounded by a geodesic loop (i.e., a geodesic segment that is also a closed curve). Geodesic loops have useful convexity properties which we will use extensively in our proof of the above theorem. For our purposes, the definition of convexity is as follows (cf. \cite{croke1988}).
	
	\begin{defin}
		Let $\gamma$ be a closed curve bounding a region $\Omega$. Then $\gamma$ is said to be convex to $\Omega$ if there is some $\epsilon>0$ such that for all $x,y\in\overline{\Omega}$ with $d(x,y)<\epsilon$, the minimizing geodesic segment between $x$ and $y$ lies within $\overline{\Omega}$. A connected region $\Omega$ is called convex if each component of $\partial\Omega$ is convex to $\Omega$.
	\end{defin} 
	
	Note that a geodesic loop is convex to a region it bounds exactly when the inward-facing angle (i.e., the angle as measured within the region) formed at its vertex is less than or equal to $\pi$.	Similarly, a closed geodesic is convex to any region it bounds.
	\par
	Convex regions are useful because they are well-behaved under the Birkhoff curve shortening process. This algorithm, a thorough analysis of which can be found in \cite{croke1988}, produces a length-reducing homotopy starting from a prescribed initial curve. The curves in this homotopy converge to either a single point, a closed geodesic, or a point at infinity. Additionally, one can fix a point $x$ of the initial curve during this process in order to produce a length-shortening homotopy that will terminate either in the point curve $x$ or a geodesic loop based at $x$. 
	Convexity of the initial curve allows us to apply the following lemma (cf. Lemma 2.2 in \cite{croke1988}).
	
	\begin{lemma}
		\label{convex_traps_loops}
		Let $\Omega$ be a convex region. Let $\gamma\subset \overline{\Omega}$ be a closed curve and let $\gamma_t$ be any curve in the homotopy produced by applying the Birkhoff curve shortening process to $\gamma$. Then $\gamma_t\subset \overline{\Omega}$. \qed
	\end{lemma}
	
	The proof of our main result considers two separate cases. If $M$ has one end, either every geodesic loop is convex to an unbounded region (which we will call being ``convex to infinity''), or there are two geodesic loops with a shared vertex point that are convex to disjoint sets. Conversely, if $M$ has two or more ends, the latter statement is always true. The proofs of these statements are established in Section \ref{secnets}. In the first case, in the proof of Lemma \ref{2_ends_core} we apply a variation of Berger's lemma along with the curve shortening process to produce a map from the round sphere to a compactified version of $M$. The existence of such a map will imply the existence of a closed geodesic of known length. In Section \ref{secnet_shortening} we deal with the second case and show that, in the absence of short closed geodesics, the curve formed by the union of the two geodesic loops with a shared vertex must shorten to a closed geodesic of bounded length.
	\par 
	Unfortunately, we do not expect that the provided bounds are sharp. An illustrative example is the singular sphere produced by gluing together two congruent equilateral triangles of height $h$. This so-called Calabi-Croke sphere is conjectured to provide the largest value of $l(M)$ relative to $\sqrt{A}$ in the case that $M$ is a sphere  (cf. \cite{calabi1992,croke1988}, see also \cite{balacheff2010,sabourau2010}). It has several shortest geodesics of length $2h$: three simple loops along its altitudes and three figure-eights. In this case, we have
	\begin{align*}
	l(M)
	=(12)^\frac{1}{4} \sqrt{A}\approx 1.316\sqrt{2A}
	\end{align*} 
	Adding a thin, infinitely long cusp to one of the vertices does not change this constant. When additional ends are added, new figure-eight geodesics are formed by wrapping around at least two ends. However, adding up to three ends to the vertices of the Calabi-Croke sphere does not decrease $l(M)$, as the new figure eights are approximately of length $(4/\sqrt{3})h$. Thus for the cases where $n\leq 3$ we still expect a sharp bound of $(12)^\frac{1}{4}\sqrt{A}$.
	
\subsection{Short Geodesic Loops}
	\label{secloops}
	
	In order to use geodesic loops to produce closed geodesics of bounded length, we need to first bound the lengths of the loops themselves. To that end, we prove the following lemma, a sharpening of the bound provided in \cite{croke1988} (cf. \cite{rotman2006}). This lemma ensures the existence of short geodesic loops encircling the ends of $M$, as described in Lemma \ref{short_loops}.
	In the lemma below and throughout this paper, $L(\gamma)$ denotes the length of the curve $\gamma$. Note also that every geodesic is assumed to be parametrized by arc length.	

\begin{lemma}
	\label{min_closed_curve}
	Suppose $M$ is a complete, orientable surface of finite area $A$. Let $x,y\in M$ and let $\tau$ be a minimizing geodesic segment from $x$ to $y$. If there exists some $w\in \tau$ such that $d(x,w)> \sqrt{A/2}$ and $d(y,w)> \sqrt{A/2}$, then there is a closed curve $\gamma$ through $w$ of length at most $\sqrt{2A}$ that is essential (i.e., not null-homotopic) in $M\setminus\{x,y\}$.
\end{lemma}

\begin{figure}[h]\color{black}
	\centering
	\def\svgwidth{1\textwidth}
	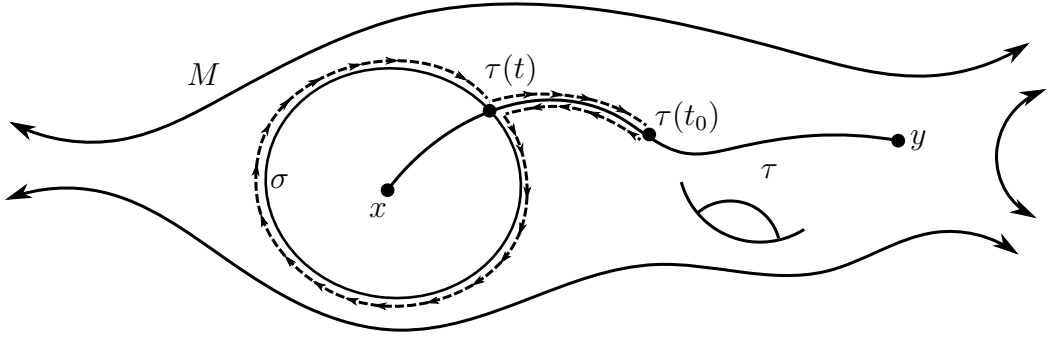
	\caption{The construction of the closed curve $\gamma$ (the dashed line).}
	\label{fig:closed_curve_construction}
\end{figure}

\begin{proof}
	Define $t_0$ such that $w=\tau(t_0)$ and hence $d(x,w)=t_0>\sqrt{A/2}$ (note that we asserted $x=\tau(0)$). 
	Therefore, by the co-area formula,
	\begin{align*}
		\int_{t_0-\sqrt{A/2}}^{t_0+\sqrt{A/2}}\left(\sqrt{2A}-2|t-t_0|\right)dt
		&= A\geq \int_{t_0-\sqrt{A/2}}^{t_0+\sqrt{A/2}}L(S(x,t))dt,
	\end{align*}
	where $S(x,t)=\{p\in M\mid d(p,x)=t\}$.
	Thus there is some $t\in(t_0-\sqrt{A/2},t_0+\sqrt{A/2})$ such that $S(x,t)$ consists of simple closed curves and $L(S(x,t))\leq \sqrt{2A}-2|t-t_0|$. Let $\sigma$ be the component of $S(x,t)$ that crosses $\tau(t)$, as in Figure \ref{fig:closed_curve_construction}. Then $\sigma$ is essential in $M\setminus\{x,y\}$, since it is either essential in $M$ or it separates $M$ into two regions, one containing $x$ and the other $y$. Therefore $\gamma=\tau\mid_{[t_0,t]}\cup\sigma\cup-\tau\mid_{[t_0,t]}$ is a closed curve passing through $w$ that is essential in $M\setminus\{x,y\}$ and has length
	\begin{align*}
		L(\gamma)=2L(\tau\mid_{[t_0,t]})+L(\sigma)\leq2|t-t_0|+\sqrt{2A}-2|t-t_0|=\sqrt{2A},
	\end{align*}
	as required.
\end{proof}
	
	In the case that $M$ contains a minimizing geodesic ray or line, we will now show that the above lemma implies the existence of short geodesic loops.
	Since we are only considering surfaces with at least one end, every point in $M$ has at least one minimizing geodesic ray emanating from it.
	For example, one can take a limit of minimizing geodesic segments (up to a subsequence) between a given point and a sequence of points that tends to infinity.  If $M$ has at least two ends, then we can similarly ensure the existence of a minimizing geodesic line.
	Therefore we can make effective use of the following lemma (cf. Lemma 3.3 in \cite{croke1988}, which only considers the cases where $M$ a plane, sphere or cylinder). Note that ``ray'' always refers to a minimizing geodesic ray, and similarly a ``line'' is a minimizing geodesic line.

\begin{lemma}
	\label{short_loops}
	Suppose $M$ is a complete, orientable surface with finite area $A$. Let $\tau$ be either a ray based at some arbitrary $x$ or a line. Pick $w\in \tau$, where if $\tau$ is a ray we require that $d(x,w)>\sqrt{A/2}$. Then there exists a geodesic loop based $w$ that is essential in $M\setminus\{x\}$ (or $M$) and intersects $\tau$ only at $w$. 
	Moreover, there is a shortest such loop which is additionally simple and has length at most $\sqrt{2A}$. If $\gamma$ and $\gamma'$ are two such loops, then $\gamma'\cap\gamma = w$.
\end{lemma}
\begin{proof}
	Let $\{x_i\},\{y_i\}$ be sequences of points on $\tau$ such that $d(x_i,w)>\sqrt{A/2}$, $d(y_i,w)>\sqrt{A/2}$, $w$ lies between $x_i$ and $y_i$, $y_i\to\tau(\infty)$, and $x_i=x$ if $\tau$ is a ray and $x_i\to\tau(-\infty)$ otherwise. Then for every $i$ we can apply Lemma \ref{min_closed_curve} to the segment of $\tau$ connecting $x_i$ to $y_i$ to obtain a closed curve $\eta_i$ that is essential in $M\setminus\{x_i,y_i\}$, passes through $w$, and has length at most $\sqrt{2A}$. Taking the limit of a convergent subsequence of these curves, we see that there exists a closed curve $\eta$ through $w$ that is essential in $M\setminus\{x\}$ (or $M$), and has length $L(\eta)\leq \sqrt{2A}$. 
	\par
	Let $U$ be the set of all closed curves $\sigma$ with $L(\sigma)\leq L(\eta)$ that pass through $w$ and are essential in $M\setminus\{x\}$ (or $M$). Since $L(\eta)/2<d(x,w)$, there is a neighbourhood of $x$ that no curves in $U$ pass through, and hence no sequence of curves in $U$ converges to a point curve. Therefore there exists a non-trivial shortest $\gamma\in U$, which is necessarily a simple geodesic loop by minimality. Moreover, $\gamma$ intersects $\tau$ only at $w$, as otherwise we could take a shortest sub-loop of $\gamma$ that intersects $\tau$ and connect it via an arc of $\tau$ to $w$ to obtain a strictly shorter curve that could be shortened to a loop with the desired properties. By a similar argument, we remark that if multiple distinct choices for $\gamma$ exist, they intersect only at $w$.
\end{proof}
\begin{corollary}
	Let $M$ be as above and let $\tau$ be a ray based at some $x$. Then there exists a shortest geodesic loop $\gamma$ through $\tau(\sqrt{A/2})$ that is essential in $M\setminus\{x\}$, intersects $\tau$ only at $\tau(\sqrt{A/2})$, is simple, and has length at most $\sqrt{2A}$. \qed
\end{corollary}

\subsection{Obtaining Convex Loops}
	\label{secnets}
	
	The fact that our surface has finite area will allow us to find pairs of loops with certain nice properties-- in particular, if $M$ has at least two ends then we can always find a pair of geodesic loops with a common vertex such that the two loops are convex to disjoint regions. Once we have established conditions for the existence of these loops in Lemma \ref{1_end_core} and Lemma \ref{2_ends_core}, we will describe an algorithm in Lemma \ref{core-filling} that will produce a closed geodesic from such a  loop pair.

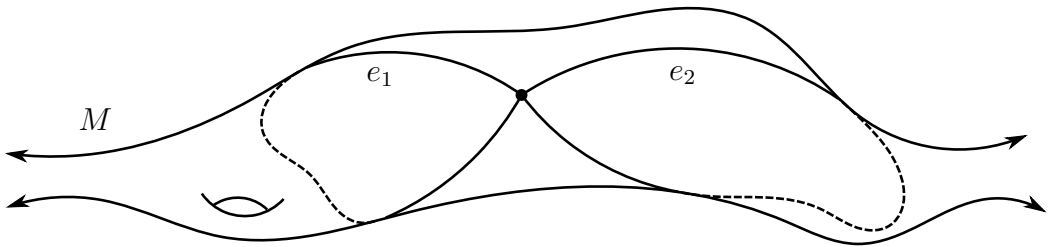
\begin{figure}[h]\color{black}
	\centering
	\def\svgwidth{1\textwidth}
\begingroup%
  \makeatletter%
  \providecommand\color[2][]{%
    \errmessage{(Inkscape) Color is used for the text in Inkscape, but the package 'color.sty' is not loaded}%
    \renewcommand\color[2][]{}%
  }%
  \providecommand\transparent[1]{%
    \errmessage{(Inkscape) Transparency is used (non-zero) for the text in Inkscape, but the package 'transparent.sty' is not loaded}%
    \renewcommand\transparent[1]{}%
  }%
  \providecommand\rotatebox[2]{#2}%
  \newcommand*\fsize{\dimexpr\f@size pt\relax}%
  \newcommand*\lineheight[1]{\fontsize{\fsize}{#1\fsize}\selectfont}%
  \ifx\svgwidth\undefined%
    \setlength{\unitlength}{295.92462726bp}%
    \ifx\svgscale\undefined%
      \relax%
    \else%
      \setlength{\unitlength}{\unitlength * \real{\svgscale}}%
    \fi%
  \else%
    \setlength{\unitlength}{\svgwidth}%
  \fi%
  \global\let\svgwidth\undefined%
  \global\let\svgscale\undefined%
  \makeatother%
  \begin{picture}(1,0.23063082)%
    \lineheight{1}%
    \setlength\tabcolsep{0pt}%
    \put(0,0){\includegraphics[width=\unitlength,page=1]{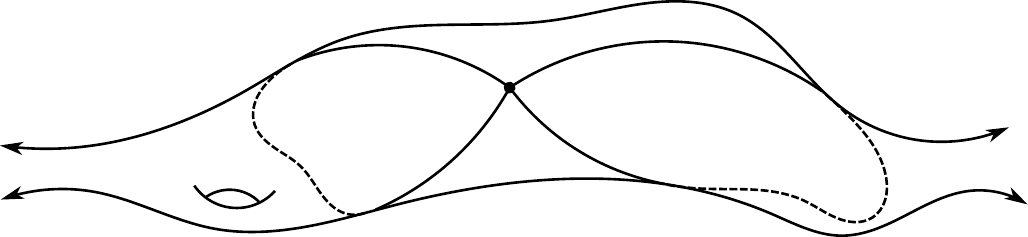}}%
    \put(0.07070678,0.11191118){\color[rgb]{0,0,0}\makebox(0,0)[lt]{\lineheight{1.25}\smash{\begin{tabular}[t]{l}$M$\end{tabular}}}}%
    \put(0.34725759,0.15898034){\color[rgb]{0,0,0}\makebox(0,0)[lt]{\lineheight{1.25}\smash{\begin{tabular}[t]{l}$e_1$\end{tabular}}}}%
    \put(0.63784386,0.16029263){\color[rgb]{0,0,0}\makebox(0,0)[lt]{\lineheight{1.25}\smash{\begin{tabular}[t]{l}$e_2$\end{tabular}}}}%
    \put(0.58337882,0.00054422){\color[rgb]{0,0,0}\makebox(0,0)[lt]{\lineheight{1.25}\smash{\begin{tabular}[t]{l} \end{tabular}}}}%
  \end{picture}%
\endgroup%

	\caption{An example of a pair of loops $e_1$ and $e_2$ with shared base point that are convex to disjoint regions.}
	\label{fig:net_example}
\end{figure}

	As a preliminary step, we recall the following topological lemma.
	 
	 \begin{lemma}
	 	\label{separating}
	 	 Let $\gamma$ be a curve in a surface $M$. If $\gamma$ does not separate $M$ into two disjoint regions, then it is homotopically non-trivial. \qed
	 \end{lemma}

	Importantly, if a curve on $M$ is not separating, then we can shorten it to a closed geodesic (note that a non-separating curve cannot escape to infinity for topological reasons). We will now build upon the arguments of Croke in \cite{croke1988} and examine the convexity properties of short geodesic loops.

\begin{lemma}
	\label{finite_area_loops_convexity}
	Suppose $M$ is a complete, orientable surface with finite area $A$. Let $\tau$ be either a ray based at arbitrary $x$ or a line, and where possible define $\gamma_t$ as a shortest geodesic loop based at $\tau(t)$ that is essential in $M\setminus\{x\}$ (or $M$). Suppose $M$ contains no geodesics of length at most $\sqrt{2A}$. Then
	\begin{enumerate}
		\item 
		If $\tau$ is a ray, either there is some $t\geq\sqrt{A/2}$ with two choices of $\gamma_t$, one convex to $x$ and one convex to $\tau(\infty)$, or every $\gamma_t$ is convex to $\tau(\infty)$.
		\item  If $\tau$ is a line, there is some $t$ with two choices of $\gamma_t$, one convex to $\tau(-\infty)$ and one convex to $\tau(\infty)$.
	\end{enumerate}
\end{lemma}
\begin{proof}
	The existence of the curves $\gamma_t$ for $t\geq\sqrt{A/2}$ such that $L(\gamma_t)\leq\sqrt{2A}$ follows from Lemma \ref{short_loops}. Every $\gamma_t$ must separate $M$ into two regions, as otherwise by Lemma \ref{separating} some $\gamma_t$ could be shortened to a geodesic of length at most $\sqrt{2A}$. Thus each $\gamma_t$ is convex to one of the regions it bounds. In particular, $\gamma_t$ is either convex to a region containing $x$ (or $\tau(-\infty)$), or to a region containing $\tau(\infty)$.
	\par
	By the co-area formula, we have
	\begin{align*}
	\infty > A \geq \int_{\sqrt{A/2}}^\infty L(\gamma_t)dt,
	\end{align*}
	and so there is some sequence $\{t_i\}$ such that $L(\gamma_{t_i})\to 0$ as $t_i\to\infty$. If some $\gamma_{t_0}$ is convex to $x$ (or $\tau(-\infty)$), then, for small $\epsilon>0$, $L(\gamma_t)<L(\gamma_{t_0})$ for $t\in(t_0-\epsilon,t_0)$. This is because cutting across the vertex of $\gamma$ with a short minimizing geodesic segment gives rise to a strictly smaller curve intersecting $\tau(t_0-\epsilon')$ for some small $\epsilon'$. Therefore if every $\gamma_t$ is convex to $x$ (or $\tau(-\infty)$), $L(\gamma_t)$ is an increasing function of $t$. Since $L(\gamma)$ is positive, this contradicts the fact that $L(\gamma_t)\to0$, so there must exist at least one $\gamma_t$ convex to $\tau(\infty)$.
	\par 
	However, the set of $t$ such that $\gamma_t$ can be chosen convex to $\tau(\infty)$ is closed, as is the set of $t$ with $\gamma_t$ convex the other way (i.e., to $x$ or $\tau(-\infty)$). Since $[\sqrt{A/2},\infty)\subset \mathbb{R}$ is connected and the first set is non-empty, either the second set is empty or the sets intersect. Thus either every $\gamma_t$ is convex to $\tau(\infty)$ or there is some $t_0$ with two choices of $\gamma_{t_0}$, one convex to each end of $\tau$.
	\par 
	If $\tau$ is a line, then $\gamma_t$ is defined for all $t\in\mathbb{R}$ and hence
	\begin{align*}
	\infty > A  \geq \int_{-\infty}^{\infty} L(\gamma_t)dt,
	\end{align*}
	so we have the additional condition that there is some sequence $\{t_j\}$ such that $L(\gamma_{t_j})\to 0$ as $t_j\to-\infty$. Thus if some $\gamma_{t_0}$ is convex to $\tau(\infty)$, then $L(\gamma_t)<L(\gamma_{t_0})$ for $t\in(t_0,t_0+\epsilon)$. As before, this leads to a contradiction if every $\gamma_t$ is convex to $\tau(\infty)$. Therefore there must be some $\gamma_t$ convex to $\tau(-\infty)$ and hence some $t_0$ with two choices of $\gamma_{t_0}$, one convex to each direction.
\end{proof}

	We now prove that if $M$ does not have a short geodesic, then we can always find a pair of loops that share a base point but are convex in different directions, even if $M$ has only one end. To eliminate the possibility that every loop is convex to a single end, we will use a few facts about the space of integral 1-cycles, denoted $Z_1(M,\bZ)$. Lying within this space is $\Gamma(M)$, the set of all 1-cycles with only one or two connected components. If we can find a non-contractible loop within $\Gamma(M)$, then the maximal length of a 1-cycle in the loop is an upper bound for the length of a shortest closed geodesic on $M$. For further details, see \cite{calabi1992}. We now describe how to produce such a loop in the one-ended case.

\begin{lemma}[The one-ended case]
	\label{1_end_core}
	Suppose $M$ is a complete, orientable surface of finite area $A$ that has one end. If $M$ contains no closed geodesics of length at most $4\sqrt{2A}$, then it contains a pair of geodesic loops with a common vertex such that the two loops are convex to disjoint regions and each has length at most $\sqrt{2A}$.
\end{lemma}
\begin{proof}
	\begin{figure}[h]
		\centering
		\def\svgwidth{0.49\textwidth}
		\raisebox{-0.5\height}{
\begingroup%
  \makeatletter%
  \providecommand\color[2][]{%
    \errmessage{(Inkscape) Color is used for the text in Inkscape, but the package 'color.sty' is not loaded}%
    \renewcommand\color[2][]{}%
  }%
  \providecommand\transparent[1]{%
    \errmessage{(Inkscape) Transparency is used (non-zero) for the text in Inkscape, but the package 'transparent.sty' is not loaded}%
    \renewcommand\transparent[1]{}%
  }%
  \providecommand\rotatebox[2]{#2}%
  \newcommand*\fsize{\dimexpr\f@size pt\relax}%
  \newcommand*\lineheight[1]{\fontsize{\fsize}{#1\fsize}\selectfont}%
  \ifx\svgwidth\undefined%
    \setlength{\unitlength}{276.72490893bp}%
    \ifx\svgscale\undefined%
      \relax%
    \else%
      \setlength{\unitlength}{\unitlength * \real{\svgscale}}%
    \fi%
  \else%
    \setlength{\unitlength}{\svgwidth}%
  \fi%
  \global\let\svgwidth\undefined%
  \global\let\svgscale\undefined%
  \makeatother%
  \begin{picture}(1,0.51315332)%
    \lineheight{1}%
    \setlength\tabcolsep{0pt}%
    \put(0,0){\includegraphics[width=\unitlength,page=1]{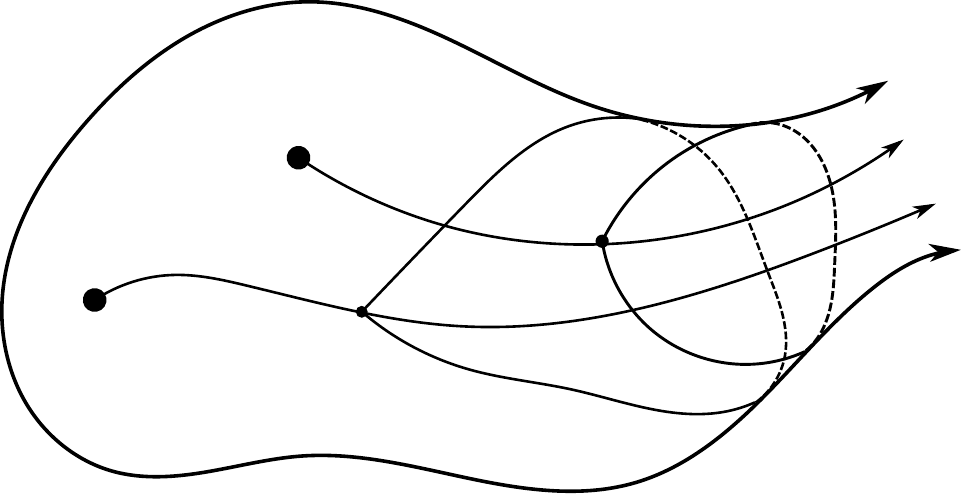}}%
    \put(0.03160502,0.21335081){\color[rgb]{0,0,0}\makebox(0,0)[lt]{\lineheight{1.25}\smash{\begin{tabular}[t]{l}$y$\end{tabular}}}}%
    \put(0.2752682,0.36803591){\color[rgb]{0,0,0}\makebox(0,0)[lt]{\lineheight{1.25}\smash{\begin{tabular}[t]{l}$x$\end{tabular}}}}%
    \put(0.62463172,0.3178364){\color[rgb]{0,0,0}\makebox(0,0)[lt]{\lineheight{1.25}\smash{\begin{tabular}[t]{l}$\gamma$\end{tabular}}}}%
    \put(0.37177532,0.31835479){\color[rgb]{0,0,0}\makebox(0,0)[lt]{\lineheight{1.25}\smash{\begin{tabular}[t]{l}$\tau$\end{tabular}}}}%
    \put(0.16972579,0.23959459){\color[rgb]{0,0,0}\makebox(0,0)[lt]{\lineheight{1.25}\smash{\begin{tabular}[t]{l}$\sigma$\end{tabular}}}}%
    \put(0.51824696,0.07509226){\color[rgb]{0,0,0}\makebox(0,0)[lt]{\lineheight{1.25}\smash{\begin{tabular}[t]{l}$\eta$\end{tabular}}}}%
    \put(0.12742765,0.07098576){\color[rgb]{0,0,0}\makebox(0,0)[lt]{\lineheight{1.25}\smash{\begin{tabular}[t]{l}$B$\end{tabular}}}}%
  \end{picture}%
\endgroup%
}
		\def\svgwidth{0.49\textwidth}
		\raisebox{-0.5\height}{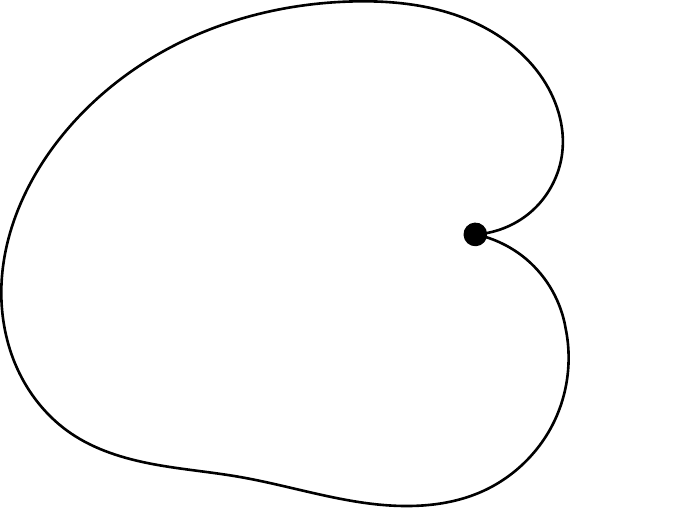}
		\caption{Left: A one-ended surface with every short loop convex to infinity.
				 Right: The geodesic segments constructed via Berger's Lemma.}
		\label{fig:one_end_construction}
	\end{figure}
	Suppose $M$ has no closed geodesics of length at most $4\sqrt{2A}$. Fix any $x\in M$ and let $\tau$ be a ray based at $x$. Define $\gamma_t$ for each $t\geq\sqrt{A/2}$ as above. We can apply Lemma \ref{finite_area_loops_convexity}, obtaining the following two cases.
	The first case is that there is some $t_0\geq\sqrt{A/2}$ with two choices of $\gamma_{t_0}$, one convex to each direction. This is our desired loop pair.
	The alternative case is that for every choice of $\tau$ and all $t\geq\sqrt{A/2}$, every $\gamma_t$ is convex to infinity. We will show that the latter case never actually occurs if $M$ contains no geodesics of length at most $4\sqrt{2A}$. 
	\par 
	First, we apply the construction given in \cite{croke1988} (see Figure \ref{fig:one_end_construction}). Define $\gamma$ as a shortest geodesic loop based at $\tau(6\sqrt{A/2})$. This curve bounds a precompact region $\Omega_\gamma$. Let $y$ be at maximum distance from $\gamma(0)$ out of all points in $\Omega_\gamma$. Let $\sigma$ be a ray from $y$ to infinity and let $\eta$ be a short loop at $\sigma(\sqrt{A/2})$. Notice that $\Omega_\gamma$ strictly contains the precompact region $\Omega_\eta$ bounded by $\eta$, as
	\begin{align*}
	d(y,\gamma)\geq d(y,\gamma(0)) -L(\gamma)/2\geq d(x,\gamma(0)) -L(\gamma)/2\geq5\sqrt{A/2}.
	\end{align*}
	\par 
	 Because $y$ is at locally maximal distance from $\gamma(0)$ (recall that $y$ lies within the interior of $\Omega_{\eta}$), we can apply a non-compact version of Berger's lemma (cf. Lemma 8.15 in \cite{comparison}). Let $\tau_1=\sigma\mid_{[0,\sqrt{A/2}]}$. Then by Berger's Lemma there are minimizing geodesics segments $\tau_2$ and $\tau_3$ from $y$ to $\gamma(0)$ such that the angle between $\tau_i$ and $\tau_{i+1}$ (denoted cyclically) at $y$ is at most $\pi$. Note that it may be the case that $\tau_3=\tau_1$ if $\tau_1$ and $\tau_2$ are segments of the same geodesic. \par 
	 Since $\gamma$ lies outside of $\Omega_{\eta}$, each $\tau_i$ intersects $\eta$. Let $\eta(t_i)=\tau_i\cap\eta$, and define the closed curves $T_{1}=\tau_1\cup\eta\mid_{[0,t_2]}\cup-\tau_2$,  $T_{2}=\tau_2\cup\eta\mid_{[t_2,t_3]}\cup-\tau_3$ and  $T_{3}=\tau_3\cup\eta\mid_{[t_3,L(\eta)]}\cup-\tau_1$.
	 Notice that each $T_i$ is convex to the precompact region $\Omega_i$ it bounds. Therefore by Lemma \ref{convex_traps_loops} and our assumption that $M$ has no geodesics of length less than $4\sqrt{2A}$, each $T_i$ shortens to a point curve $q_i\in \Omega_i$. Similarly, $\eta$ must shorten to a point at infinity. Therefore the neighbourhood of infinity bounded by $\eta$ is, topologically, a half-cylinder, as the presence of a handle would obstruct the shortening of $\eta$. Therefore $M$ is homeomorphic to the plane, since it is covered by the regions induced by shortening $\eta$, $T_1$, $T_2$, and $T_3$. Thus we can compactify $M$ by adding a point at infinity to produce a topological sphere, denoted $\hat{M}$.
	\par 
	Recall that $\Gamma(\hat{M})$ is the set of all 1-cycles in $\hat{M}$ with only one or two connected components. We now define the following homotopy in $\Gamma(\hat{M})$:
	\begin{enumerate}
		\item $(\{q_1\}\cup\{q_2\})\sim (T_{1}\cup T_{2})$ by reversing the curve shortening process.
		\item $(T_{1}\cup T_{2})\sim (\tau_1\cup\eta\mid_{[0,t_3]}\cup-\tau_3)$ by retracting  $-\tau_2\cup\tau_2$ to its base point $\eta(t_2)$ along itself.
		\item $(\tau_1\cup\eta\mid_{[0,t_3]}\cup-\tau_3)\sim (\tau_1\cup\eta\cup-\eta\mid_{[t_3,L(\eta)]}\cup -\tau_3)=(\eta\cup -T_3)$ by extending $\eta\mid_{[t_3,L(\eta)]}\cup- \eta\mid_{[t_3,L(\eta)]}$ from its base point $\eta(t_3)$ along itself.
		\item  $(\eta\cup -T_3) \sim (\{\infty\}\cup\{q_3\})$ by the curve shortening homotopy.
		\item $(\{\infty\}\cup\{q_3\})\sim(\{q_1\}\cup\{q_2\})$
	\end{enumerate} 
	This homotopy produces a closed loop $\Phi$ in $\Gamma(\hat{M})\subset Z_1(\hat{M},\bZ)$. Notice that the longest cycle in this homotopy (e.g., $T_{1}\cup T_{2}$) has length at most $4\sqrt{2A}$. Thus if our loop in $Z_1(\hat{M},\bZ)$ is not contractible, $M$ contains a closed geodesic of length at most $4\sqrt{2A}$ (cf. \cite{calabi1992, nabutovsky2004}). By the work of Almgren \cite{almgren1962}, we know that $Z_1(\hat{M},\bZ)\simeq H_2(\hat{M},\bZ)=\bZ$, so it is enough to check the that homology class corresponding to $\Phi$ is non-zero. The class corresponding to this homotopy under Almgren's isomorphism is represented by the 2-cycle determined by the union of the disks generated by shortening $T_1,T_2,T_3$ and $\eta$ in $\hat{M}$. These disks naturally glue together to produce a 2-sphere, and hence this cycle represents a non-trivial element of $H_2(\hat{M},\bZ)\simeq H_2(S^2,\bZ)$. Therefore the corresponding path in $Z_1(\hat{M},\bZ)$ cannot be trivial, either. For further details, see the analogous final proof in \cite{nabutovsky2002}. Thus we obtain a closed geodesic of length at most $4\sqrt{2A}$, contradicting our initial assumption. 
\end{proof}

\begin{lemma}[The two-ended case]
	\label{2_ends_core}
	Suppose $M$ is a complete, orientable surface of finite area $A$ that has at least two ends. If $M$ contains no closed geodesics of length at most $\sqrt{2A}$, then it contains a pair of geodesic loops with a common vertex such that the two loops are convex to disjoint regions and each has length at most $\sqrt{2A}$.
\end{lemma}

\begin{proof}
	Suppose $M$ contains no closed geodesics of length at most $\sqrt{2A}$. Let $\tau$ be a line in $M$. Then by Lemma \ref{finite_area_loops_convexity} there must be a point $t_0$ such that there are two shortest geodesic loops with vertex at $\tau(t_0)$, each convex to a different end of $M$. This is our desired pair of loops.
\end{proof}

\subsection{Shortening Pairs of Loops} 
	\label{secnet_shortening}
	
	We now define a curve shortening technique that will, if successfully applied to our loops with shared vertex, induce a deformation retraction of the sphere to a point. In particular, we will show that if $M$ has no short closed geodesics, then it must be homeomorphic to a punctured sphere. Our process then gives rise to a retraction of the compactification of $M$ (which will be a sphere) to a point curve derived from shortening our loops. This is an impossibility, as the sphere is not contractible. The resultant contradiction will prove the existence of a short closed geodesic on $M$ whose length is bounded by the sum of the lengths of the original loops. 
	
\begin{lemma}
	\label{core-filling}
	Suppose $M$ is a complete surface with finite area $A$. If $M$ contains a pair of parametrized geodesic loops $e_1,e_2$ with a common vertex such that the two loops are convex to disjoint regions, then $M$ also contains a closed geodesic of length at most $L(e_1)+L(e_2)$.
\end{lemma}
\begin{proof}
	Suppose that $M$ has two loops $e_1$ and $e_2$ as described in the statement of the lemma but contains no closed geodesics of length at most $L(e_1)+L(e_2)$. We will derive a contradiction. As a preliminary step, orient the parametrized curves $e_1$ and $e_2$ such that $e_1\cup e_2$ can be parametrized as a single transversely self-intersecting loop. Note that, by assumption, $e_1$ and $e_2$ intersect only at their shared vertex. Define the loop $e_3=e_1\cup -e_2$, such that $e_3$ (unlike $e_1\cup e_2$) can be transformed into a simple curve by a short homotopy. For each $i$ let $\Omega_i$ be the region to which $e_i$ is convex. Note that $\Omega_3=(\Omega_1\cup\Omega_2)^c$.
	\par 
	First, we show that our loops lie in the complement of a locally convex set $V$ (i.e., each connected component of $V$ is convex). This fact will allow us to ``quotient out'' the ends of $M$ during our shortening algorithm. Each $\Omega_i$ must be a disk or half-cylinder, as otherwise $e_i$ could be shortened to a closed geodesic of length at most $L(e_i)$. If $\Omega_i$ is a half-cylinder, then let $\tau_i$ be a ray from the vertex of $e_i$ to infinity in $\Omega_i$. Choose a shortest loop $\gamma_i$ at $\tau(t)$ for $t$ large enough that $\gamma_i$ and $e_i$ do not intersect. If $\gamma_i$ was not convex to $\tau_i(\infty)$, then the region bounded by $\gamma_i$ and $e_i$ would be an annulus with convex boundary and hence would contain a short closed geodesic by Lemma \ref{convex_traps_loops} (e.g., by shortening $\gamma_i$). Thus $\gamma_i$ must be convex to $\tau_i(\infty)$. Apply this process for $i\in\{1,2,3\}$ whenever $\Omega_i$ is a half-cylinder. If $M$ is compact, then every $\Omega_i$ is a disk and we can take $V=\emptyset$. Otherwise, the region bounded by the (non-empty) union of the $\gamma_i$ contains our loops and is the complement of a locally convex set $V$.
	\par
	We will now show that we can apply an iterative curve shortening procedure to our loops to construct a continuous family $\{M_t\}$ of topological spheres terminating in a point. This contradicts the fact that the sphere is not contractible, proving that $M$ does in fact contain a short closed geodesic.
	\par
	First, we apply the curve shortening process to the loop $e_1\cup e_2$. At each $t>0$, the curve at time $t$ in the shortening homotopy can still be viewed as two (possible non-simple or constant) loops $e_{1,t}$ and $e_{2,t}$ with a common base point, which can be chosen continuously in time. Since $M$ contains no short geodesics, $e_{1,t}\cup e_{2,t}$ either converges to a point or escapes to infinity. In the first case, let $t_f$ be the first time at which $e_{1,t}\cup e_{2,t}$ is a point. In the second case, Lemma \ref{convex_traps_loops} and the convexity of $V$ together imply that we can define some $t_f$ such that $e_{1,t}\cup e_{2,t} \subset V$ for all $t\geq t_f$.
	\par 
	We construct the family $\{M_t\}$ in the following way (see Figure \ref{fig:net_shortening}).
	\begin{enumerate}
		\item 
		As discussed above, apply the curve shortening process to the self-intersecting curve $e_{1}\cup e_{2}$ to obtain a continuous family of loop pairs $\{e_{1,t},e_{2,t}\}$ with shared base point for $t\in[0,t_f]$. As before, let $e_{3,t}$ be the concatenation $e_{1,t}\cup - e_{2,t}$, so that its self-intersection at the loops' base point can be removed by a short homotopy.
		\item 
		For each $t\in[0,t_f]$, apply the curve shortening process to the three loops $e_{i,t}$. Since $M$ contains no short geodesics, for every $t\geq0$ each $e_{i,t}$ shortens to a point or escapes to infinity. The homotopy produced by the curve shortening process generates a surface $\Sigma_{i,t}$ bounded by the loop $e_{i,t}$, which is either a half-sphere, a half-cylinder or a point.
		\item 
		Glue the three surfaces $\Sigma_{i,t}$ together along their shared boundary $e_{1,t}\cup e_{2,t}$. This produces a (possibly multiply-punctured) sphere $S_t$.
		\item 
		Let $\sim$ be the equivalence relation defined on $M$ such that $x\sim y$ exactly when $x$ and $y$ are in the same connected component of $V$. This relation collapses each component of $V$ to a single point. We define $M_t = S_t/\mathord\sim$. 
	\end{enumerate}

	\begin{figure}[h]
		\centering
		\def\svgwidth{0.49\textwidth}
\begingroup%
  \makeatletter%
  \providecommand\color[2][]{%
    \errmessage{(Inkscape) Color is used for the text in Inkscape, but the package 'color.sty' is not loaded}%
    \renewcommand\color[2][]{}%
  }%
  \providecommand\transparent[1]{%
    \errmessage{(Inkscape) Transparency is used (non-zero) for the text in Inkscape, but the package 'transparent.sty' is not loaded}%
    \renewcommand\transparent[1]{}%
  }%
  \providecommand\rotatebox[2]{#2}%
  \newcommand*\fsize{\dimexpr\f@size pt\relax}%
  \newcommand*\lineheight[1]{\fontsize{\fsize}{#1\fsize}\selectfont}%
  \ifx\svgwidth\undefined%
    \setlength{\unitlength}{195.64044735bp}%
    \ifx\svgscale\undefined%
      \relax%
    \else%
      \setlength{\unitlength}{\unitlength * \real{\svgscale}}%
    \fi%
  \else%
    \setlength{\unitlength}{\svgwidth}%
  \fi%
  \global\let\svgwidth\undefined%
  \global\let\svgscale\undefined%
  \makeatother%
  \begin{picture}(1,0.68040613)%
    \lineheight{1}%
    \setlength\tabcolsep{0pt}%
    \put(0,0){\includegraphics[width=\unitlength,page=1]{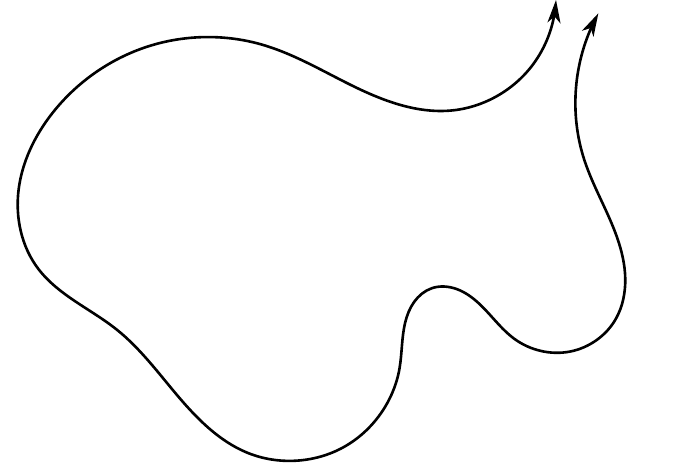}}%
    \put(0.08408397,0.3484696){\color[rgb]{0,0,0}\makebox(0,0)[lt]{\lineheight{1.25}\smash{\begin{tabular}[t]{l}$e_{1,t}$\end{tabular}}}}%
    \put(0.76192975,0.36103211){\color[rgb]{0,0,0}\makebox(0,0)[lt]{\lineheight{1.25}\smash{\begin{tabular}[t]{l}$e_{2,t}$\end{tabular}}}}%
    \put(-0.00323457,0.57707645){\color[rgb]{0,0,0}\makebox(0,0)[lt]{\lineheight{1.25}\smash{\begin{tabular}[t]{l}$M$\end{tabular}}}}%
    \put(0,0){\includegraphics[width=\unitlength,page=2]{net_shortening_example_tex.pdf}}%
  \end{picture}%
\endgroup%

		\def\svgwidth{0.49\textwidth}
\begingroup%
  \makeatletter%
  \providecommand\color[2][]{%
    \errmessage{(Inkscape) Color is used for the text in Inkscape, but the package 'color.sty' is not loaded}%
    \renewcommand\color[2][]{}%
  }%
  \providecommand\transparent[1]{%
    \errmessage{(Inkscape) Transparency is used (non-zero) for the text in Inkscape, but the package 'transparent.sty' is not loaded}%
    \renewcommand\transparent[1]{}%
  }%
  \providecommand\rotatebox[2]{#2}%
  \newcommand*\fsize{\dimexpr\f@size pt\relax}%
  \newcommand*\lineheight[1]{\fontsize{\fsize}{#1\fsize}\selectfont}%
  \ifx\svgwidth\undefined%
    \setlength{\unitlength}{222.56621827bp}%
    \ifx\svgscale\undefined%
      \relax%
    \else%
      \setlength{\unitlength}{\unitlength * \real{\svgscale}}%
    \fi%
  \else%
    \setlength{\unitlength}{\svgwidth}%
  \fi%
  \global\let\svgwidth\undefined%
  \global\let\svgscale\undefined%
  \makeatother%
  \begin{picture}(1,0.78702253)%
    \lineheight{1}%
    \setlength\tabcolsep{0pt}%
    \put(0,0){\includegraphics[width=\unitlength,page=1]{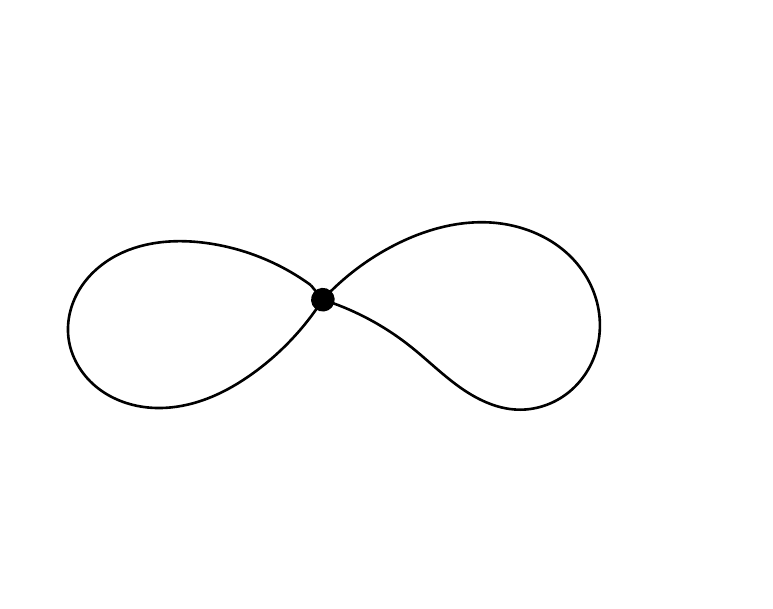}}%
    \put(-0.00284325,0.33801676){\color[rgb]{0,0,0}\makebox(0,0)[lt]{\lineheight{1.25}\smash{\begin{tabular}[t]{l}$e_{1,t}$\end{tabular}}}}%
    \put(0.79073118,0.34320692){\color[rgb]{0,0,0}\makebox(0,0)[lt]{\lineheight{1.25}\smash{\begin{tabular}[t]{l}$e_{2,t}$\end{tabular}}}}%
    \put(0,0){\includegraphics[width=\unitlength,page=2]{net_shortening_tex.pdf}}%
    \put(0.55839813,0.09365997){\color[rgb]{0,0,0}\makebox(0,0)[lt]{\lineheight{1.25}\smash{\begin{tabular}[t]{l}$S_t$\end{tabular}}}}%
  \end{picture}%
\endgroup%

		\caption{The construction of the punctured sphere $S_t$ for some fixed $t$.}
		\label{fig:net_shortening}
	\end{figure}
	\par 
	By assumption, the sets $\Sigma_{i,0}=\Omega_i$ are disjoint (except for their shared boundary) and hence the union of the images of the three initial homotopies covers $M$ without overlaps. Therefore $S_0\simeq M$ in a natural way, and hence $M_0\simeq M/\mathord\sim$. Moreover, the family $M_t$ is continuous. This is because the curve shortening process is continuous on compact surfaces, and we are only defining $M_t$ in the finite time interval $[0,t_f]$. For all $t\in[0,t_f)$, the surface $M_t$ is homeomorphic to $S^2$, as each connected component of $V$ is collapsed to a point under the quotient map. However, $M_{t_f}$ is a point, because either $e_{1,t_f}\cup e_{2,t_f} $ itself is a point, or $e_{1,t_f}\cup e_{2,t_f} \subset V$ and hence by convexity the image of each $e_{i,t_f}$ under the shortening process lies entirely in $V$. Thus we have arrived at a contradictory conclusion. Therefore, some step of our procedure must have been obstructed by the presence of a closed geodesic, which will necessarily have length at most $L(e_1)+L(e_2)$.
\end{proof}

\subsection{Main Theorem}
\label{secmain_theorem}

	We now collect the above cases and prove our main result.

\begin{theorem}
	\label{main_theorem}
	Suppose $M$ is a complete, orientable surface with finite area $A$ and $n$ ends. Let $l(M)$ be the length of a shortest closed geodesic on $M$.
	\begin{enumerate}
		\item If $n\leq1$, then $l(M)\leq 4\sqrt{2A}$.
		\item If $n\geq2$, then $l(M)\leq 2\sqrt{2A}$.
	\end{enumerate}
\end{theorem}
	
\begin{proof}
	The case where $n=0$ is due to the previous results described at the beginning of this paper (e.g., \cite{rotman2006, croke1988}). 
	If $n=1$, by Lemma \ref{1_end_core} we know that either $l(M)\leq4\sqrt{2A}$ or $M$ admits a pair of geodesic loops each of length at most $\sqrt{2A}$ with a common vertex such that the two loops are convex to disjoint regions. In the second case, we apply Lemma \ref{core-filling} to obtain a closed geodesic of length at most $2\sqrt{2A}$. Therefore it must be the case that $l(M)\leq 4\sqrt{2A}$.
	If $n\geq2$, then by Lemma \ref{2_ends_core} either $l(M)\leq\sqrt{2A}$ or the hypothesis of Lemma \ref{core-filling} is satisfied by a pair of loops that each have length less than $\sqrt{2A}$. Thus we again conclude that $M$ has a closed geodesic of length at most $2\sqrt{2A}$, and hence $l(M)\leq 2\sqrt{2A}$.
\end{proof}
\begin{corollary}
	Suppose $M$ is a complete, non-orientable surface. Let $\tilde{M}$ be its orientable double cover bestowed with the covering metric. Suppose $\tilde{M}$ has finite area $\tilde{A}$ and $\tilde{n}$ ends. Let $l(M)$ be the length of a shortest closed geodesic on $M$.
	\begin{enumerate}
		\item If $\tilde{n}\leq1$, then $l(M)\leq 4\sqrt{2\tilde{A}}$.
		\item If $\tilde{n}\geq2$, then $l(M)\leq 2\sqrt{2\tilde{A}}$.
	\end{enumerate}
\end{corollary}
\begin{proof}
	By applying the proof above we can find a short geodesic on $\tilde{M}$ that will be mapped by the orientation covering to a geodesic on $M$ of at most equal length. 
\end{proof}

\bibliography{geodesic}

\begin{thebibliography}{21}
\providecommand{\natexlab}[1]{#1}
\providecommand{\url}[1]{\texttt{#1}}
\expandafter\ifx\csname urlstyle\endcsname\relax
  \providecommand{\doi}[1]{doi: #1}\else
  \providecommand{\doi}{doi: \begingroup \urlstyle{rm}\Url}\fi

\bibitem[Almgren(1962)]{almgren1962}
F.~J. Almgren.
\newblock The homotopy groups of the integral cycle groups.
\newblock \emph{Topology}, 1\penalty0 (4):\penalty0 257--299, 1962.

\bibitem[Asselle and Mazzucchelli(2017)]{assellemazzucchelli2016}
L.~Asselle and M.~Mazzucchelli.
\newblock On the existence of infinitely many closed geodesics on non-compact
  manifolds.
\newblock \emph{Proceedings of the American Mathematical Society}, 145\penalty0
  (6):\penalty0 2689--2697, 2017.

\bibitem[Balacheff(2010)]{balacheff2010}
F.~Balacheff.
\newblock A local optimal diastolic inequality on the two-sphere.
\newblock \emph{Journal of Topology and Analysis}, 02\penalty0 (01):\penalty0
  109--121, 2010.

\bibitem[Bangert(1980)]{bangert1980}
V.~Bangert.
\newblock Closed geodesics on complete surfaces.
\newblock \emph{Mathematische Annalen}, 251\penalty0 (1):\penalty0 83--96,
  1980.

\bibitem[Benci and Giannoni(1991)]{bencigiannoni1991}
V.~Benci and F.~Giannoni.
\newblock Closed geodesics on noncompact {Riemannian} manifolds.
\newblock \emph{Comptes Rendus de l'Académie des Sciences Paris - Series I -
  Mathematics}, 321\penalty0 (11):\penalty0 857--861, 1991.

\bibitem[Burago and Zalgaller(1988)]{burago1980}
Y.~D. Burago and V.~A. Zalgaller.
\newblock \emph{Geometric Inequalities}.
\newblock Springer-Verlag Berlin Heidelberg, 1988.

\bibitem[Burns and Matveev(2013)]{burns2013}
K.~Burns and V.~S. Matveev.
\newblock Open problems and questions about geodesics.
\newblock art. arXiv:1308.5417, 2013.

\bibitem[Calabi and Cao(1992)]{calabi1992}
E.~Calabi and J.~G. Cao.
\newblock Simple closed geodesics on convex surfaces.
\newblock \emph{Journal of Differential Geometry}, 36\penalty0 (3):\penalty0
  517--549, 1992.

\bibitem[Cheeger and Ebin(1976)]{comparison}
J.~Cheeger and D.~G. Ebin.
\newblock \emph{Comparison theorems in {Riemannian} geometry}.
\newblock Amsterdam: North-Holland Pub. Co., 1976.

\bibitem[Croke and Katz(2003)]{CrokeKatz2003}
C.~Croke and M.~Katz.
\newblock Universal volume bounds in {Riemannian} manifolds.
\newblock In S.-T. Yau, editor, \emph{Surveys in Differential Geometry},
  volume~8, pages 109--138. International Press, 2003.

\bibitem[Croke(1988)]{croke1988}
C.~B. Croke.
\newblock Area and the length of the shortest closed geodesic.
\newblock \emph{J. Differential Geom.}, 27\penalty0 (1):\penalty0 1--21, 1988.

\bibitem[Gromov(1983)]{Gromov1983}
M.~Gromov.
\newblock Filling {Riemannian} manifolds.
\newblock \emph{Journal of Differential Geometry}, 18\penalty0 (1):\penalty0
  1--147, 1983.

\bibitem[Hebda(1982)]{hebda1982}
J.~Hebda.
\newblock Some lower bounds for the area of surfaces.
\newblock \emph{Inventiones mathematicae}, 65\penalty0 (3):\penalty0 485--490,
  1982.

\bibitem[Nabutovsky and Rotman(2002)]{nabutovsky2002}
A.~Nabutovsky and R.~Rotman.
\newblock The length of the shortest closed geodesic on a 2-dimensional sphere.
\newblock \emph{International Mathematics Research Notices}, 2002\penalty0
  (23):\penalty0 1211--1222, 2002.

\bibitem[Nabutovsky and Rotman(2004)]{nabutovsky2004}
A.~Nabutovsky and R.~Rotman.
\newblock Volume, diameter and the minimal mass of a stationary 1-cycle.
\newblock \emph{Geometric {\&} Functional Analysis GAFA}, 14\penalty0
  (4):\penalty0 748--790, Aug 2004.

\bibitem[Rotman(2006)]{rotman2006}
R.~Rotman.
\newblock The length of a shortest closed geodesic and the area of a
  2-dimensional sphere.
\newblock \emph{Proceedings of the American Mathematical Society}, 134\penalty0
  (10):\penalty0 3041--3047, 2006.

\bibitem[Rotman(2011)]{rotman2011}
R.~Rotman.
\newblock Flowers on {Riemannian} manifolds.
\newblock \emph{Mathematische Zeitschrift}, 269:\penalty0 543--554, 2011.

\bibitem[Rotman(2019)]{rotman2019}
R.~Rotman.
\newblock Wide short geodesic loops on closed {Riemannian} manifolds.
\newblock art. arXiv:1910.01772, Oct 2019.
\newblock Preprint.

\bibitem[Sabourau(2004)]{sabourau2004}
S.~Sabourau.
\newblock Filling radius and short closed geodesics of the 2-sphere.
\newblock \emph{Bulletin de la Société Mathématique de France}, 132\penalty0
  (1):\penalty0 105--136, 2004.

\bibitem[Sabourau(2010)]{sabourau2010}
S.~Sabourau.
\newblock Local extremality of the {Calabi–Croke} sphere for the length of
  the shortest closed geodesic.
\newblock \emph{Journal of the London Mathematical Society}, 82\penalty0
  (3):\penalty0 549--562, 2010.

\bibitem[Thorbergsson(1978)]{thorbergsson1978}
G.~Thorbergsson.
\newblock Closed geodesics on non-compact {Riemannian} manifolds.
\newblock \emph{Mathematische Zeitschrift}, 159\penalty0 (3):\penalty0
  249--258, Oct 1978.

\end{thebibliography}

\end{document}